\documentclass{amsart}
\usepackage{xspace}

\def\makedef#1{\expandafter\gdef\csname#1\endcsname}
\def\mat#1{\ensuremath{#1}\xspace}
\def\dmat#1#2{\gdef#1{\mat{#2}}} 				% \dmat\name{commands}
\def\oper#1{\makedef{#1}{\operatorname{#1}}}

\newif\ifrem\remtrue

\theoremstyle{plain}
\newtheorem{nthr}{Theorem}[section]
\newtheorem{theorem}		[nthr]{Theorem}

\newtheorem{lemma}		[nthr]{Lemma}
\newtheorem{corollary}	[nthr]{Corollary}
\newtheorem{remark}       [nthr]{Remark}

\oper{Exp}
\oper{Log}
\oper{Pow}
\oper{pow}
\oper{End}
\oper{Aut}
\oper{Hom}
\oper{GL}
\oper{SL}
\oper{PGL}
\oper{Map}
\oper{Spec}

\def\eql#1#2{\begin{equation}\label{#2}#1\end{equation}}
\def\set#1{\{#1\}}
\def\sets#1#2{\{#1\mid#2\}}
\def\rbr#1{\left(#1\right)}
\def\n#1{\left\lvert#1\right\rvert}
\def\pser#1{[\![#1]\!]} %formal power series [[#1]]
\def\lser#1{(\!(#1)\!)} %formal Laurent power series [[#1]]
\def\overst#1#2{\mat{\substack{#1\\#2}}} % #1 over #2

\def\bN{\mathbb{N}}
\def\bZ{\mathbb{Z}}
\def\bQ{\mathbb{Q}}
\def\bF{\mathbb{F}}
\def\bC{\mathbb{C}}
\dmat\cP{\mathcal P}
\dmat\cX{\mathcal X}

\def\de{\delta}
\dmat\la\lambda
\def\si{\sigma}
\def\ta{\tau}
\dmat\vi\varphi
\def\Ga{\Gamma}
\def\fm{\mathfrak{m}}

\def\mto{\mapsto}
\def\inv{^{-1}}
\def\xx{\times}
\def\iso{\simeq}
\def\bop{\bigoplus}
\def\irr{\mathrm{irr}}
\def\ind{\mathrm{ind}}
\def\sb{\subset}
\def\GIT{/\!\!/}

\usepackage{hyperref}
\begin{document}
\title[Character varieties]{Arithmetic of character varieties of free groups}
\author{Sergey Mozgovoy}%
\author{Markus Reineke}%
\email{mozgovoy@maths.tcd.ie}
\email{mreineke@uni-wuppertal.de}
\date{\today}
\begin{abstract}
Via counting over finite fields, we derive explicit formulas for the $E$-polynomials and Euler characteristics of $\GL_d$- and $\PGL_d$-character varieties of free groups. We prove a positivity property for these polynomials and relate them to the number of subgroups of finite index.
\end{abstract}

\maketitle

\section{Introduction and statement of the results}\label{introduction}

Given a finitely generated group $\Ga$ and a complex reductive algebraic group $G$, one can consider the character variety
$$X_\Ga(G)=\Hom(\Ga,G)\GIT G=\Spec\rbr{\bC[\Hom(\Ga,G)]^G},$$
the spectrum of the (finitely generated) ring of $G$-invariant functions on the space of representations of $\Ga$ in $G$. 
The study of these varieties, in particular for $\Ga$ the fundamental group of a manifold, is a common theme in geometry and topology. Arithmetic aspects of character varieties have been studied very fruitfully e.g.~in~\cite{hausel_mixed}.

From now on let $\Ga=F_m$ be the free group in $m\geq 0$ generators (i.e.~the fundamental group of an $(m+1)$-punctured sphere). The character varieties $X_\Ga(G)$ for the groups $G=\GL_d(\bC)$, $G=\PGL_d(\bC)$, and $G=\SL_d(\bC)$ were studied e.g.\ in \cite{florentino_topology,florentino_singularities}, and from an arithmetic point of view in \cite{cavazos_e}.
We summarize some basic geometric properties of these representation varieties. For $m\geq 2$, the variety $X_\Ga(\GL_d(\bC))$ is an irreducible affine variety of dimension $(m-1)d^2+1$. It is usually a singular variety; its smooth locus typically reduces to $X^\irr_\Ga(\GL_d(\bC))$, the subset corresponding to irreducible representations.

Basic to the arithmetic study of character varieties is the following approach: there exists a $\bZ$-model $\cX_d$ of $X_\Ga(\GL_d(\bC))$ using Seshadri's GIT over base rings \cite{seshadri_geometric}. Namely, consider the scheme
$$\cX_d
=\Spec\rbr{\bZ[\Hom(\Ga,\GL_d(\bZ))]^{\GL_d(\bZ)}}$$
over $\Spec(\bZ)$. Then
$$\Spec(\bC)\xx_{\Spec(\bZ)}\cX_d\simeq X_\Ga(\GL_d(\bC)).$$
We find a similar open subscheme $\cX^{\irr}_d\subset\cX_d$ of irreducible representations (over some open subscheme of $\Spec(\bZ)$).

We can thus reduce to finite fields and consider the counting functions
$$A_d(q)=\n{\cX_d(\bF_q)},\qquad
A^{\irr}_d(q)=\n{\cX_d^{\irr}(\bF_q)}$$
defined on prime powers $q$. By definition, $A_d(q)$ (resp.~$A^\irr_d(q)$) equals the number of isomorphism classes of completely reducible representations (resp.~of absolutely irreducible representations, that is, representations that stay irreducible under any finite field extension) of the group $F_m$ of dimension $d$ over $\bF_q$.
%\rem{Relation is as between abs. indecomposable reps and all reps, not just sums of abs. indecomposables}
Our first main result is the following:

\begin{theorem}\label{t11}
The varieties $X_\Ga(\GL_d(\bC))$ and $X^\irr_\Ga(\GL_d(\bC))$ are polynomial count \cite[Appendix]{hausel_mixed}, that is, 
the functions $A_d(q)$ and $A_d^\irr(q)$ are polynomials with integer coefficients in $q$.
Consequently, the $E$-polynomials of these varieties are
\[E(X_\Ga(\GL_d(\bC)); u,v)=A_d(uv),\qquad
E(X_\Ga^\irr(\GL_d(\bC)); u,v)=A_d^\irr(uv).\]
\end{theorem}

Similar formulas for the group $\PGL_d(\bC)$ are proved in Corollary \ref{cm}. The above theorem immediately follows from explicit formulas for $A_d(q)$ and $A_d^\irr$ which we now formulate. Consider the formal power series ring $\bQ(q)\pser t$ with the maximal ideal $\fm$. We define the $\bQ(q)$-linear shift operator $T$ on $\bQ(q)\pser t$ by
\eql{T(t^d)=q^{(1-m)\binom d2}t^d.}{T}
We define the plethystic exponential $\Exp:\fm\to 1+\fm$ by
\eql{\Exp(q^it^d)=(1-q^it^d)^{-1}\qquad
\text{and}\qquad
\Exp(f+g)=\Exp(f)\Exp(g).}{Exp}
It admits an inverse $\Log:1+\fm\to\fm$. Define the power operator
\cite{mozgovoy_computational}
\[\Pow:(1+\fm)\xx\bQ(q)\pser t\to 1+\fm,\qquad
\Pow(f,g)=\Exp(g\Log(f)).\]
Finally, define the series
\eql{F(t)=\sum_{d\geq 0}\rbr{(q-1)\ldots(q^d-1)}^{m-1}t^d\in\bQ(q)\pser t.}{F}

\begin{theorem}\label{t12}
In $\bQ(q)\pser t$, we have
$$\sum_{d\ge0}A_d(q)t^d=\Pow(T^{-1}F(t)^{-1},1-q),$$
$$\sum_{d\ge1}A_d^\irr(q)t^d=(1-q)\Log(T^{-1}F(t)^{-1}).$$
\end{theorem}

Using combinatorial methods, we prove a quite surprising positivity property of the counting polynomials $A_d(q)$:

\begin{theorem}\label{t13}
We have $A_d(q)\in\bN[q-1]$.
\end{theorem}

Contrarily, $A_d^\irr(q)$ does not fulfill this property. It is tempting to interpret this result geometrically, for example as a prediction for a paving of $X_\Ga(\GL_d(\bC)$ by tori (or at least quotients by finite groups of tori); but such a paving can not be expected to be very natural since it cannot be compatible with $X_\Ga^\irr(\GL_d(\bC))$. It would be interesting to test this prediction for the rather explicitly known $\GL_2(\bC)$-character varieties.
We can determine explicitly the lowest order Taylor coefficient of $A_d(q)$ around $q=1$ to derive:

\begin{theorem}\label{t14}
For $m\geq 2$, the topological Euler characteristic of the varieties $X_\Ga(\PGL_d(\bC))$ and $X_\Ga^\irr(\PGL_d(\bC))$ are given by
$$\chi(X_\Ga(\PGL_d(\bC)))=\varphi(d)d^{m-2},$$
$$\chi(X_\Ga^\irr(\PGL_d(\bC)))=\mu(d)d^{m-2},$$
where $\varphi$ and $\mu$ denote the Euler totient function, resp.~the M\"obius function.
\end{theorem}

Again, it would be interesting to have a geometrical explanation for these Euler characteristics.
%\newpage
We can also count absolutely indecomposable representations (that is, representations that stay indecomposable under any finite field extension) of~$F_m$ over finite fields:

\begin{theorem}%\label{th:}
For any $d\ge1$, there exists a polynomial $A_d^\ind(q)\in\bZ[q]$ that counts isomorphism classes of absolutely indecomposable representations of $F_m$ of dimension $d$ over $\bF_q$ for every prime power $q$. These polynomials satisfy
\[\sum_{d\ge1}A_d^\ind(q)t^d=(q-1)\Log\rbr{\sum_{\la}r_\la^{m-1}t^{\n\la}},\]
where the sum ranges over all partitions and, for any partition \la,
\[r_\la=\prod_{n\ge1}q^{\la_n^2}(q^{-1})_{\la_{n}-\la_{n+1}},\qquad (q)_n=(1-q)\dots(1-q^n).\]
\end{theorem}

In contrast to \cite{cavazos_e}, our approach to the arithmetic of the character varieties is non-geometric and purely formal. It is based on the Hall algebra approach to the arithmetic of moduli spaces of representations of quivers developed in \cite{mozgovoy_number}. In fact, we can directly adopt the methods there to derive the explicit formula for $A_d(q)$. The other results follow by a detailed study of the right hand side of this formula and explicit combinatorics. One can expect this approach to even work in a motivic Hall algebra \cite{bridgeland_introduction}, leading to a formula similar to Theorem \ref{t12} for the motives of character varieties.

The paper is organized as follows: in Section \ref{motive}, we recall the Hall algebra methods of \cite{mozgovoy_number} to prove Theorems \ref{t11} and \ref{t12}. In Section \ref{eulerchar}, we derive Theorem \ref{t14} using an elementary argument. Combinatorial notions are introduced in Section \ref{positivity} to derive Theorem \ref{t13}.
In Section \ref{all_reps}, we derive a formula for the number of absolutely indecomposable representations of $F_m$ over finite fields.
In Section \ref{subgroups}, we explain how the numbers of subgroups of fixed index in free groups \cite{hall_subgroups} can be reconstructed from the counting polynomials, providing a potential link of the present study with the topic of subgroup growth \cite{lubotzky_subgroup}.

\section{Computation of the counting polynomial}\label{motive}
Let $\Ga=F_m$ be the free group in $m$ generators and let $k$ be a field. A representation of the group algebra $k\Ga$ can be identified with a representation of the free associative algebra $A$ with $m$ generators such that all generators act bijectively on the representation. Therefore the category of $k\Ga$-representations can be identified with an exact subcategory of the abelian category of $A$-representations. Since the algebra $A$ is a path algebra of the quiver with one vertex and $m$ loops,
the methods of \cite{mozgovoy_number} for the explicit calculation of the number of isomorphism classes of absolutely irreducible representations of quivers continue to work for $\Ga$-representations. We recall the main steps of this calculation and refer to \cite{mozgovoy_number} for the proofs which hold without any changes.
%Since the group algebra $\bC\Ga$ is a localization of the free associative algebra in $m$ generators (or, in other words, since $\bC\Ga$ is a formally smooth algebra \cite{kontsevich_noncommutative}) 

In the following, let $k$ be a finite field with $q$ elements. We first define the Hall algebra $H\lser{k\Ga}$ of the group algebra $k\Ga$. As a (complete, $\bZ_{\geq 0}$-graded) $\bQ$-vector space, we define
$$H\lser{k\Ga}=\prod_{[V]}\bQ\cdot[V],$$
where the direct product ranges over all isomorphism classes of (finite-dimensional) representations $V$ of $k\Ga$. We have a natural grading by the dimension $\dim V$. We define a product on $H\lser{k\Ga}$ by
$$[V]\cdot[W]=\sum_{[X]}g_{V,W}^X[X],$$
where $g_{V,W}^X$ equals the number of subrepresentations $U\subset X$ such that $U$ is isomorphic to $W$ and $X/U$ is isomorphic to $V$. Then we have \cite[3.3]{mozgovoy_number}:

\begin{lemma} This product endows $H\lser{k\Ga}$ with a structure of a $\bZ_{\geq 0}$-graded complete local associative unital $\bQ$-algebra. In particular, every element with constant term $1$ (the class of the zero-dimensional representation) is invertible in $H\lser{k\Ga}$.
\end{lemma}

We define an evaluation map
$$I:H\lser{k\Ga}\to\bQ\pser{t},\qquad
[V]\mto\frac{1}{|\Aut(V)|}t^{\dim V}.$$

By \cite[Lemma 3.4]{mozgovoy_number}, we have:

\begin{lemma}\label{lm1}
The composition $T\circ I:H\lser{k\Ga}\to\bQ\pser{t}$ is a $\bQ$-algebra homomorphism, where the operator $T$ is defined by \eqref{T}.
\end{lemma}

We consider the (invertible) element
$$e=\sum_{[V]}[V]\in H\lser{k\Ga}.$$
Using $\Hom(\Ga,\GL_d(k))\simeq\GL_d(k)^{m}$ and
\[\n{\GL_d(k)}=\prod_{i=0}^{d-1}(q^d-q^i)
=q^{\binom d2}\prod_{i=1}^d(q^i-1),\]
the following is easily verified:

\begin{lemma}\label{lm2} We have
$$I(e)=T^{-1}F(t),\qquad I(e\inv)=T\inv F(t)\inv,$$
where $F(t)$ is defined by \eqref{F}.
\end{lemma}
\begin{proof}
By the definitions, we have
\[I(e)
=\sum_{d\ge0}\frac{\n{\GL_d(k)^m}}{\n{\GL_d(k)}}t^d
=\sum_{d\ge0}\rbr{q^{\binom d2}\prod_{i=1}^d(q^i-1)}^{m-1}t^d
=T\inv F(t).
\]
This implies $TI(e)=F(t)$ and therefore $TI(e\inv)=F(t)\inv$ by Lemma \ref{lm1}.
Therefore $I(e\inv)=T\inv F(t)\inv$.
\end{proof}

The key lemma towards counting absolutely irreducible representations is \cite[Lemma 3.5]{mozgovoy_number}:

\begin{lemma}
Writing $$e^{-1}=\sum_{[V]}\gamma_V[V]$$ in $H\lser{k\Ga}$, we have  the following description of the coefficients $\gamma_V$
\begin{enumerate}
	\item If $V$ is not completely reducible, then $\gamma_V=0$.
	\item If $V=\bigoplus_{[S]}S^{m_S}$ is completely reducible (the direct sum ranging over all isomorphism classes of irreducible representations $S$ of $k\Ga$), then
$$\gamma_V=\prod_{[S]}(-1)^{m_S}|\End(S)|^{\binom{m_S}2}.$$
\end{enumerate}
\end{lemma}

Based on this lemma, one can prove (see \cite[Theorem 4.2]{mozgovoy_number}) that:

\begin{theorem}\label{th:3} %lm3
The function $A_d^\irr(q)$ is a polynomial in $q$ and we have in $\bQ\pser{t}$
\[I(e^{-1})=\Exp\rbr{\frac{1}{1-q}\sum_{d\geq 1}A_d^\irr(q)t^d}_{q=\n k}.\]
%\rem{We can not take Exp of a function in $q$. We need to know that it is a polynomial in $q$}
\end{theorem}

\begin{proof}[Proof of Theorems \ref{t11} and \ref{t12}]
We obtain from Theorem \ref{th:3} and Lemma \ref{lm2} that $A_d^\irr(q)$ is a polynomial in $q$ satisfying in $\bQ(q)\pser t$
\[\sum_{d\ge1}A_d^\irr(q)t^d=(1-q)\Log\rbr{T\inv F(t)\inv}.\]
This implies that $A_d^\irr(q)$ has integer coefficients.
Using \cite[Lemma 5]{mozgovoy_computational}, one can prove that
$A_d(q)$ is a polynomial in $q$ satisfying in $\bQ(q)\pser{t}$
$$\sum_{d\geq 0}A_d(q)t^d
=\Exp\rbr{\sum_{d\geq 1}A_d^{\irr}(q)t^d}.$$
Therefore
$$\sum_{d\geq 0}A_d(q)t^d
=\Exp\rbr{(1-q)\Log\rbr{T\inv F(t)\inv}}
=\Pow\rbr{T\inv F(t)\inv,1-q}$$
and $A_d(q)$ has integer coefficients.
This finishes the proof of Theorem \ref{t12}. To prove Theorem \ref{t11}, we note that by \cite[Appendix]{hausel_mixed}, the $E$-polynomials are given by the counting polynomials evaluated at $q=uv$.
\end{proof}

Now we pass to the $\PGL_d(\bC)$-character varieties. There is a free action of the torus $(\bC^*)^m$ on $X_\Ga(\GL_d(\bC))$. The quotient fibration $X_\Ga(\GL_d(\bC))\to X_\Ga(\PGL_d(\bC))$ has fibres isomorphic to $(\bC^*)^m$ and is Zariski locally trivial by Hilbert $90$.
Similarly for the open subsets corresponding to irreducible representations.

%acting by dilation determinant induces a map
%$$\det:\Hom(\Ga,\GL_d(\bC))\to\Hom(\Ga,\bC^*)\simeq(\bC^*)^m$$
%with fibre $\Hom(\Ga,\PGL_d(\bC))$ over $(1,\ldots,1)$.
%It is invariant for the $\GL_d(\bC)$-action on $\Hom(\Ga,\GL_d(\bC))$, thus descends to a map
%$$\det:X_\Ga(\GL_d(\bC))\to (\bC^*)^m$$
%with special fibre $X_\Ga(\PGL_d(\bC))$.
%
%This map being equivariant for the action of the torus $(\bC^*)^m$ acting by dilation, resp.~by multiplication by the $d$-th power, we see that it realizes $X_\Ga(\GL_d(\bC))$ as the total space of a (by Hilbert 90) Zariski-locally trivial fibration over $(\bC^*)^m$ with special fibre $X_\Ga(\PGL_d(\bC))$. 

Consequently, we see that the $E$-polynomials of the $\PGL_d(\bC)$-character varieties are known:

\begin{corollary}\label{cm} We have
$$E(X_\Ga(\PGL_d(\bC)); u,v)=\frac{1}{(uv-1)^m}A_d(uv),$$
$$E(X_\Ga(\PGL_d(\bC))^\irr; u,v)=\frac{1}{(uv-1)^m}A_d^\irr(uv).$$
\end{corollary}

%xand 
%which is a fibration with special fibre $Y_d$, thus
%$$[Y_d]=\frac{[X_d]}{(q-1)^m}$$
%(if $m\not=0$), and similarly
%$$[Y_d^\irr]=\frac{[X_d^\irr]}{(q-1)^m}$$

%$$\sum_{d\geq 0}\frac{[\Hom(\Ga,\GL_d(k))]}{[\GL_d(k)]}t^d\circ\Exp(\frac{\sum_{d\geq 1}[X_d^\irr]t^d}{1-q})=1,$$
%where the twisted multiplication in $R\pser{t}$ is defined via $t^d\circ t^e=q^{(m-1)de}t^{d+e}$.\\

%Since a representation of $\GammA_d$ in $\GL_d(k)$ is just specified by an $m$-tuple of elements of $\GL_d(k)$, we have
%$$\Hom(\Ga,\GL_d(k))\simeq\GL_d(k)^m.$$
%Moreover, we can pass to the usual multiplication in $R\pser{t}$ via the operator $T$ mapping $t^d$ to $q^{-(m-1){d\choose 2}}t^d$, namely, we have $$T(x\circ y)=T(x)\cdot T(y).$$
%Together with
%$$[\GL_d(k)]=(q^d-1)(q^d-q)\ldots(q^d-q^{d-1}),$$
%this yields the following result:
%\begin{equation}\label{e1}T\Exp(\frac{\sum_{d\geq 1}[X_d^\irr]t^d}{1-q})=(\sum_{d\geq 0}((q^d-1)\ldots(q-1))^{m-1}t^d)^{-1}.\end{equation}
%Moreover, we have by \cite{mozgovoy_number}:
%\begin{equation}\label{e2}\sum_{d\geq 0}[X_d]t^d=\Exp(\sum_{d\geq 1}[X_d^\irr]).\end{equation}

%The three previous formulas allow us to compute $[X_d^\irr]$, $[X_d]$ and $[Y_d]$ recursively, in particular, we see that all these motives belong to the subring of $R$ generated by $q$.

We conclude this section with an example.
First, we have
\[A^\irr_1(q)=A_1(q)=(q-1)^m.\]
As the first nontrivial special case of Theorem \ref{t12}, we have $A_2^{\irr}(q)=$
$$(q-1)^m\rbr{q^{m-1}(q-1)^{m-1}((q+1)^{m-1}-1)-\frac{1}{2}(q+1)^{m-1}+\frac{1}{2}(q-1)^{m-1}}.$$
This implies
\begin{multline*}
A_2(q)=A_2^{\irr}(q)+\frac12\rbr{A_1^\irr(q^2)+A_1^\irr(q)^2}=\\
%=A_2^{\irr}(q)+\frac{1}{2}(q-1)^m((q+1)^m+(q-1)^m)\\
(q-1)^m\rbr{q^{m-1}(q-1)^{m-1}\rbr{(q+1)^{m-1}-1}+\frac12q\rbr{(q+1)^{m-1}+(q-1)^{m-1}}}
\end{multline*}
and thus
\begin{multline*}
E(X_\Ga(\PGL_2(\bC));u,v)=\frac{1}{(uv-1)^m}A_2(uv)\\
=(uv)^{m-1}(uv-1)^{m-1}\rbr{(uv+1)^{m-1}-1}+\frac{1}{2}uv\rbr{(uv+1)^{m-1}+(uv-1)^{m-1}},
\end{multline*}
which should be compared with \cite[Theorem B]{cavazos_e}.

\section{Euler characteristic}\label{eulerchar}

%consider the specialization of the motives of $Y_d$ resp.~$Y_d^\irr$ at $q=1$, which is just the Euler characteristic in cohomology with compact support:

%{\bf Theorem:} For $m\geq 2$ and $d\geq 1$, we have
%$$\chi(Y_d^\irr)=\mu(d)d^{m-2}$$
%and
%$$\chi(Y_d)=\varphi(d)d^{m-2},$$
%where $\mu(d)$ and $\varphi(d)$ denote the Moebius function, resp. the Euler totient function.\\

In this section, we prove Theorem \ref{t14}.
We first need a lemma on the behaviour of specialization at $q=1$ for special plethystic exponentials and logarithms. Let $A$ be the subring of $\bQ(q)$ consisting of all rational functions without pole at $q=1$.

\begin{lemma}\label{lec}
Assume that
	\[\Log\rbr{1+(q-1)^m\sum_{n\ge1}a_n(q)t^n}=(q-1)^m\sum_{n\ge1}b_n(q)t^n\in\bQ(q)\pser t.\]
Then $a_n(q)\in A$ for all $n\ge1$ if and only if $b_n(q)\in A$ for all $n\ge 1$. Moreover,
\[b_n(1)=\sum_{d\mid n}a_{n/d}(1)\mu(d)d^{m-1},\qquad a_n(1)=\sum_{d\mid n} b_{n/d}(1)d^{m-1}.\]
\end{lemma}

\begin{proof} We recall a more direct definition of the operators $\Exp$ and $\Log$ on formal power series: for $n\ge1$, define the $\bQ$-linear Adams operator $\psi_n$ on $R=\bQ(q)\pser t$ by $\psi_n(q^it^j)=q^{ni}t^{nj}$, and define $\Psi=\sum_{n\geq 1}\frac{1}{n}\psi_n$ with inverse $\Psi^{-1}=\sum_{n\ge1}\frac{\mu(n)}{n}\psi_n$. Then $\Exp=\exp\circ\Psi$ and $\Log=\Psi^{-1}\circ\log$.

Using this description,we immediately derive the following formula:

$$\Log\rbr{1+\sum_{n\ge1}r_n(q)t^n}
=\sum_{n=1}^\infty\sum_{ij=n}\frac{\mu(i)}{i}\sum_{j=c_1+\ldots+c_l}\frac{(-1)^{l-1}}{l}
\prod_{k=1}^lr_{c_k}(q^i)t^n.$$
Now if $r_n(q)=(q-1)^ma_n(q)$, then
\begin{multline*}
\sum_{n\ge1}b_n(q)t^n=(q-1)^{-m}\Log\rbr{1+(q-1)^m\sum_{n\ge1}a_n(q)t^n}\\
=\sum_{n=1}^\infty\sum_{ij=n}\frac{\mu(i)}{i}\sum_{j=c_1+\ldots+c_l}\frac{(-1)^{l-1}}{l}\frac{(q^i-1)^{ml}}{(q-1)^m}
\prod_{k=1}^la_{c_k}(q^i)t^n.
\end{multline*}
We see that the summand corresponding to the decomposition $j=c_1+\dots+c_l$ has a zero of order at least $m(l-1)$ at $q=1$.
Specializing this formula at $q=1$, we see that only terms with $l=1$(and thus $c_1=j$) contribute. Thus the formula simplifies to
$$b_n(1)=\sum_{ij=n}\frac{\mu(i)}{i}i^ma_j(1)
=\sum_{ij=n}a_j(1)\mu(i)i^{m-1},$$
and the first claim follows.

Similarly, we prove the second claim. Using the above description of $\Exp$ we derive the following formula:
$$\Exp\rbr{\sum_{n\geq 1}r_n(q)t^n}
=1+\sum_{n=1}^\infty\sum_{n=c_1+\ldots+c_l}\frac{1}{l!}\sum_{(i_kj_k=c_k)_k}\prod_{k=1}^l\frac{r_{j_k}(q^{i_k})}{i_k}t^n.$$
Now if $r_n(q)=(q-1)^mb_n(q)$, this formula yields
\begin{multline*}
\sum_{n\ge1}a_n(q)t^n
=\frac{1}{(q-1)^m}\rbr{\Exp\rbr{(q-1)^m\sum_{n\geq 1}b_n(q)t^n}-1}\\
=\sum_{n=1}^\infty\sum_{n=c_1+\ldots+c_l}\frac{1}{l!}\sum_{(i_kj_k=c_k)_k}\prod_{k=1}^l\frac{b_{j_k}(q^{i_k})(q^{i_k}-1)^m}{i_k}\frac{1}{(q-1)^m}t^n.
\end{multline*}
As soon as $l\geq 2$ in a summand on the right hand side, this summand specializes to zero at $q=1$. Thus, after this specialization, we only have to consider summands with $l=1$, and thus $c_1=n$, which reads
$$a_n(1)=\sum_{ij=n}\frac{b_j(1)}{i}\left.\frac{(q^i-1)^m}{(q-1)^m}\right|_{q=1}
=\sum_{ij=n}b_j(1)i^{m-1},$$
as claimed.
\end{proof}

We can now prove Theorem \ref{t14}:

\begin{proof}
We can write the series $T^{-1}F(t)^{-1}$ of the previous sections in the form
\[T^{-1}F(t)^{-1}=1+\sum_{d\ge1}(q-1)^{d(m-1)}a_d(q)t^d\]
for some $a_d(q)\in\bQ[q]$ with $a_1(q)=-1$. By Theorem \ref{t12}, we have
\[\frac{1}{(q-1)^m}\sum_{d\ge1}A_d^\irr(q)t^d
=-\frac1{(q-1)^{m-1}}\Log\rbr{1+\sum_{d\ge1}(q-1)^{d(m-1)}a_d(q)t^d}\]
and applying Lemma \ref{lec}, we obtain
$$\left.\frac{1}{(q-1)^m}A_d^\irr(q)\right|_{q=1}
=-a_1(1)\mu(d)d^{m-2}=\mu(d)d^{m-2},$$
which, together with Corollary \ref{cm}, proves the first part of the theorem. Setting $b_d(q)=(q-1)^{-m}A_d^{\irr}(q)$ we can write, by Theorem \ref{t12},
\[\sum A_d(q)t^d=\Exp\rbr{(q-1)^m\sum_{d\ge1}b_d(q)t^d}\]
and applying Lemma \ref{lec} again, we derive
\begin{multline*}
\left.\frac{1}{(q-1)^m}A_d(q)\right|_{q=1}
=\sum_{ij=d}b_j(1)i^{m-1}
=\sum_{ij=d}\mu(j)j^{m-2}i^{m-1}\\
=\sum_{ij=d}\mu(j)i\cdot d^{m-2}
=\varphi(d) d^{m-2},
\end{multline*}
proving the second part of the theorem, again by Corollary \ref{cm}.
\end{proof}

\section{Positivity}\label{positivity}
The goal of this section is to prove Theorem \ref{t13}, that is, that the polynomials $A_d(q)$ determined in the previous sections satisfy
\eql{A_d\in\bN[s],\qquad s=q-1.}{goal}
By Theorem \ref{t12}, we have
\eql{\sum_{d\ge0}A_dt^d=\Pow\rbr{T\inv F(t)\inv,1-q},}{eq:AF}
where the series $F(t)$, defined in \eqref{F}, can be written in the form
\eql{F(t)=\sum_{d\ge0}[d]_q^{!(m-1)}\rbr{(q-1)^{m-1}t}^d,}{F alt}
with $[d]_q^!=\prod_{i=1}^d\frac{q^i-1}{q-1}$.
We will prove positivity in several steps.

\subsection{Positivity of the inverse}
Let
\eql{F(t)\inv=1-\sum_{n\ge1}a_n(q)t^n.}{eq:a}
We claim that $a_n\in\bN[s]$, where $s=q-1$.
We will prove actually a stronger result.

Let $S_n$ be the group of permutations of $[n]=\set{1,\dots,n}$. For any $\si\in S_n$, let $l(\si)$ denote its length. It can be described as a number of inversions
\[l(\si)=\n{\sets{i<j}{\si(i)>\si(j)}}.\]
Let $G_n=S_n^{m-1}$ and, for any $\si=(\si_1,\dots,\si_{m-1})\in G_n$, let
\[l(\si)=l(\si_1)+\dots+l(\si_{m-1}).\]
For any $n\ge1$, let $P_n\sb G_n$ be the set of connected elements, that is, elements $\si=(\si_1,\dots,\si_{m-1})$ such that there is no subinterval $[k]$ for $k<n$, fixed by all~$\si_i$.

\begin{theorem}%\label{th:}
Let
\[\rbr{\sum_{n\ge0}[n]_q^{!(m-1)}t^n}\inv=1-\sum_{n\ge1} a_nt^n.\]
Then
\[a_n=\sum_{\si\in P_n}q^{l(\si)}\in\bN[q].\]
\end{theorem}
\begin{proof}
It is known \cite[Ch.~III, Eq.~1.3.viii]{macdonald_symmetric}
that
\[\sum_{\si\in S_n}q^{l(\si)}=[n]^!_q.\]
This implies that
\[\sum_{n\ge0}[n]_q^{!(m-1)}t^n=\sum_{n\ge0}\sum_{\si\in G_n}q^{l(\si)}t^n.\]
The theorem will be proved if we will show that
\[\rbr{\sum_{n\ge0}\sum_{\si\in G_n}q^{l(\si)}t^n}\rbr{1-\sum_{n\ge1}\sum_{\si\in P_n}q^{l(\si)}t^n}=1\]
or equivalently
\[\sum_{\si\in G_n}q^{l(\si)}=\sum_{k+l=n,l\ge1}\sum_{\si\in G_k}\sum_{\ta\in P_l}q^{l(\si)+l(\ta)}.\]
Any element in $G_n$ can be uniquely written in the form $(\si,\ta)\in G_k\xx P_l$ where $k+l=n$, $k\ge0$, and $l\ge1$. It is clear that the length of $(\si,\ta)$ is equal to $l(\si)+l(\ta)$ and the theorem follows.
\end{proof}
Applying this theorem to \eqref{F alt}, we see that the polynomials $a_n$ determined by \eqref{eq:a} are contained in $\bN[s]$. The same is then true
if we substitute $F(t)\inv$ by $T\inv F(t)\inv$.

\subsection{Positivity of the power}\label{s:pos}
The goal of this section is to prove the following result

\begin{theorem}\label{th:pos}
Let
\[f=1-\sum_{n\ge1} a_n(q)t^n\in\bQ[q]\pser t,\] where $a_n\in\bQ_{\ge0}[q-1]$. Then
\[\Pow(f,1-q)\in\bQ_{\ge0}[q-1]\pser t.\]
\end{theorem}

Applying this result to $f=T\inv F(t)\inv$, we prove \eqref{goal} and therefore prove Theorem \ref{t13}.
In order to prove the above theorem we will use the formula for $\Pow(f,1-q)$ proved in \cite[Lemma 22]{mozgovoy_computational}. Let $\Phi_d(q)$ be the number of irreducible monic polynomials of degree $d$  over $\bF_q$ with a nonzero constant coefficient . Then
\[\psi_n(q-1)=q^n-1=\sum_{d\mid n}d\Phi_d(q),\]
\[\Phi_n(q)=\frac1n\sum_{d\mid n}\mu(n/d)(q^d-1)
=\frac1n\sum_{d\mid n}\mu(n/d)q^d-\de_{n1}.\]
%Let $\Phi_d^*(q)=\Phi_d(q)-\de_{d1}$. Then
%\[\psi_n(q-1)=q^n-1=\sum_{d\mid n}d\Phi^*_d(q).\]
It is proved in \cite[Lemma 22]{mozgovoy_computational} that
\eql{\Pow(f,1-q)=\prod_{d\ge1}\psi_d(f)^{-\Phi_d},}{Pow}
where on the right we use $f^g=\pow(f,g)=\exp(g\log(f))$.
We will show that for $f$ as in the theorem, each multiple on the right is in $\bQ_{\ge0}[s]\pser t$, where $s=q-1$, and therefore $\Pow(f,1-q)\in\bQ_{\ge0}[s]\pser t$ as required. 

\begin{remark}%\label{rk:}
Note that $\psi_n(s)=(s+1)^n-1\in\bN[s]$. Therefore, if $f\in\bQ_{\ge0}[s]$, then $\psi_n(f)\in\bQ_{\ge0}[s]$.
\end{remark}

\begin{remark}%\label{rk:}
Computer tests show that if $f\in\bQ_{\ge0}[s]\pser t$, then $\Pow(f,q-1)\in\bQ_{\ge0}[s]\pser t$. This is slightly different from our statement. Our strategy of the proof will not work in this case as the multiples on the right of the form $\pow(f,s^k)$ are not in $\bQ_{\ge0}[s]\pser t$ in general.
\end{remark}

%\begin{remark}%\label{rk:}
%We have for example
%\[\Pow(1-t,1-q)=\Exp((1-q)(-t))=\frac{1-t}{1-qt}
%=1+\sum_{k\ge1}(q^k-q^{k-1})t^k
%=1+\sum_{k\ge1}s(s+1)^{k-1}t^k.
%\]
%\end{remark}

\begin{lemma}%\label{lm:}
For any $n\ge1$, we have
\[n\Phi_n(q)\in\bN[q-1].\]
\end{lemma}
\begin{proof}
We will use the idea from \cite{butler_number}.
For $n\ge2$, we have
\[n\Phi_n(s+1)
=\sum_{d\mid n}\mu(d)(s+1)^{n/d}
=\sum_{k\ge0}\sum_{d\mid n}\mu(d)\binom{n/d}ks^k.\]
For any $d\mid n$, we have
\begin{multline*}
\binom{n/d}k=\n{\set{1\le a_1<\dots<a_k\le n/d}}\\
=\n{\sets{1\le a_1<\dots<a_k\le n}{d\mid\gcd(a_1,\dots,a_k)}}.
\end{multline*}
Therefore
\begin{multline*}
\sum_{d\mid n}\mu(d)\binom{n/d}k
=\sum_{d\ge1}\mu(d)\cdot\n{\sets{1\le a_1<\dots<a_k\le n}{d\mid\gcd(a_1,\dots,a_k,n)}}\\
=\n{\sets{1\le a_1<\dots<a_k\le n}{\gcd(a_1,\dots,a_k,n)=1}}.
\end{multline*}
Indeed, for any tuple $1\le a_1<\dots<a_k\le n$ with $\gcd(a_1,\dots,a_k)=m$, its contribution to the second sum is
\[\sum_{d\mid\gcd(m,n)}\mu(d)=\de_{\gcd(m,n),1}.\]
\end{proof}

%\begin{remark}%\label{rk:}
%Let $Q$ be a quiver and $s_{\al,k}(q)$ be the number of irreducible $Q$-representations $M$ over $\bF_q$ with $\dim M=\al$ and $\dim\End(M)=k$. Let $a_\al(q)$ be the number of absolutely irreducible representations. Then
%\[\psi_r(a_\al)=\sum_{d\mid r}ds_{d\al,d}.\]
%In particular, for a quiver with one vertex and one loop, we have $a_n(q)=\de_{n1}q$ and $s_{n,d}(q)=\de_{nd}\Phi_n(q)$. It is possible that $s_{n,d}\in\bN[q-1]$ also for an $m$-loop quiver. This would imply $a_n=s_{n,1}\in\bN[q-1]$.
%\end{remark}

\begin{lemma}%\label{lm:}
Let
\[f=1-\sum_{n\ge1}a_n(q)t^n\in\bQ(q)\pser t,\]
where $a_n\in\bQ_{\ge0}[s]$ and let $g\in\bQ_{\ge0}[s]$. Then
\[\pow(f,{-g})\in\bQ_{\ge0}[s]\pser t.\]
\end{lemma}
\begin{proof}
It is enough to show that $\pow(1-t,-g)\in\bQ_{\ge0}[s]\pser t$. But
\[\pow(1-t,-g)=\exp(-g\log(1-t))=\exp\rbr{g\sum_{n\ge1}\frac{t^n}n}\in\bQ_{\ge0}[s]\pser t.\]
\end{proof}

Applying this lemma to $\psi_d(f)^{-\Phi_d}$ for $d\ge1$ and using Formula \eqref{Pow}, we prove Theorem \ref{th:pos}.

\section{Counting indecomposable representations}\label{all_reps}
Let $k$ be a finite field.
For any $d\ge 0$, let $G_d=\GL_d(k)$ and $R_d=\Hom(\Gamma_m,G_d)$. Then $G_d$ acts on $R_d$ by conjugation.
Using the Kac-Stanley-Hua approach \cite{hua_counting,mozgovoy_computational}, we will prove a formula for the number of $G_d$-orbits in $R_d$:

\begin{theorem}%\label{th:}
We have
\[\sum_{d\ge0}\n{R_d/G_d}t^d=\Pow\rbr{\sum_{\la}r_\la(q)^{m-1}t^{\n\la},q-1}_{q=\n k},\]
where the sum ranges over all partitions and, for any partition \la,
\[r_\la(q)=\prod_{n\ge1}q^{\la_n^2}(q^{-1})_{\la_{n}-\la_{n+1}},\qquad (q)_n=(1-q)\dots(1-q^n).\]
\end{theorem}
\begin{proof}
Using the Burnside formula, we can write
\eql{\n{R_d/G_d}=\sum_{[g]\in G_d/\sim}\n{R_d^g}/\n{G_d^g}=\sum_{g\in G_d/\sim}\n{G_d^g}^{m-1},}{burn}
where the sum ranges over the conjugacy classes, $R_d^g$ is the set of $g$-invariant elements in $R_d$, and $G_d^g$ is the centralizer of $g$.

Let us describe the set of conjugacy classes of $G_d$.
Let $\Phi$ be the set of all monic irreducible polynomials in a variable $t$ over $k$ with a nonzero constant coefficient.
The Jordan blocks of invertible matrices are parametrized by pairs $(n,f)$, where $n\ge1$ and $f\in\Phi$. The Jordan block $J(n,f)$ corresponds to the action of $x$ on $k[x]/(f^n)$ and has size $n\cdot \deg f$. Conjugacy classes in $G_d$ are parametrized by maps
\[\vi:\bN_{>0}\xx \Phi\to\bN\]
such that $\sum_{f\in\Phi}\deg f\sum_{n\ge1}n\vi(n,f)=d$. Equivalently, we can consider $\vi$ as a map $\Phi\to\cP$, where $\cP=\Map_0(\bN_{>0},\bN)$ is the set of maps with finite support. The set \cP can be identified with the set of all partitions, where to any $m\in\cP$ we associate the usual partition $\la$ with $\la_n=\sum_{i\ge n}m_i$.
Note that $\n\la=\sum_{n\ge1}\la_n=\sum_{n\ge1}nm_n$.

The conjugacy class corresponding to $\vi:\Phi\to\cP$ is given by the action of $x$ on
\[V_\vi=\bop_{f\in\Phi}\bop_{n\ge 1}(k[x]/(f^n))^{\vi(n,f)}.\]
Its centralizer has order \cite[Theorem 2.1]{mozgovoy_motivicb}
\[\n{\Aut(V_\vi)}=\n{\End(V_\vi)}\prod_{f\in\Phi}\prod_{n\ge1}(q^{-\deg f})_{\vi(n,f)}, \]
where $q=\n k$,
\[\dim\End(V_\vi)=\sum_{f\in\Phi}\deg f\sum_{k,l\ge1}\min\set{k,l}\vi(k,f)\vi(l,f)
=\sum_{f\in\Phi}\deg f\sum_{n\ge1}\la(f)_n^2\]
and $\la(f)$ is the partition corresponding to $\vi(-,f)\in\cP$.
This implies
\begin{multline*}
\n{\Aut(V_\vi)}
=\prod_{f\in\Phi}\prod_{n\ge1}q^{\deg f\cdot \la(f)_n^2}(q^{-\deg f})_{\vi(n,f)}\\
=\prod_{f\in\Phi}\psi_{\deg f}\rbr{\prod_{n\ge1}q^{\la(f)_n^2}(q^{-1})_{\la(f)_n-\la(f)_{n+1}}}
=\prod_{f\in\Phi}\psi_{\deg f}\rbr{r_{\la(f)}(q)}
,
\end{multline*}
where $\psi_n$ is the Adams operation. Therefore, applying \eqref{burn}, we obtain
\[\n{R_d/G_d}
=\sum_{\overst{\la:\Phi\to\cP}{\sum_f\deg f\cdot\n{\la(f)}=d}}
\prod_{f\in\Phi}\psi_{\deg f}\rbr{r_{\la(f)}(q)^{m-1}}.
\]
This implies
\begin{multline*}
\sum_{d\ge1}\n{R_d/G_d}t^d
=\sum_{\la:\Phi\to\cP}
\rbr{\prod_{f\in\Phi}\psi_{\deg f}\rbr{r_{\la(f)}(q)^{m-1}}}t^{\sum_{f\in\Phi}\deg f\cdot\n{\la(f)}}\\
=\sum_{\la:\Phi\to\cP}
\prod_{f\in\Phi}\psi_{\deg f}\rbr{r_{\la(f)}(q)^{m-1}t^{\n{\la(f)}}}
=\prod_{f\in\Phi}\psi_{\deg f}\rbr{\sum_{\la\in\cP}r_\la(q)^{m-1}t^{\n\la}}\\
=\prod_{d\ge1}\psi_{d}\rbr{\sum_{\la\in\cP}r_\la(q)^{m-1}t^{\n\la}}^{\Phi_d(q)},
\end{multline*}
where the $\Phi_d(q)$ were defined in Section \ref{s:pos}.
Considering the $\Phi_d(q)$ as polynomials in $q$ and applying Formula \eqref{Pow}, we obtain
\[\prod_{d\ge1}\psi_{d}\rbr{\sum_{\la\in\cP}r_\la(q)^{m-1}t^{\n\la}}^{\Phi_d(q)}
=\Pow\rbr{\sum_{\la\in\cP}r_\la(q)^{m-1}t^{\n\la},q-1}.
\]
\end{proof}

For any $d\ge1$ let $R_d^\ind\sb R_d$ denote the set of absolutely indecomposable representations. This subset is invariant under the action of $G_d$ and the quotient can be identified with the set of isomorphism classes of $d$-dimensional absolutely indecomposable representations.

\begin{corollary}\label{cr:ind}
We have
\[\sum_{d\ge0}\n{R_d^\ind/G_d}t^d=(q-1)\Log\rbr{\sum_{\la}r_\la(q)^{m-1}t^{\n\la}}_{q=\n k},\]
where the sum ranges over all partitions.
\end{corollary}
\begin{proof}
For any finite field $\bF_q$, let
\[M_d(q)=\n{R_d(\bF_q)/G_d(\bF_q)},\qquad
A_d^\ind(q)=\n{R_d^\ind(\bF_q)/G_d(\bF_q)}.\]
We know from the previous theorem that $M_d(q)$ is a polynomial in $q$. One can prove \cite[Lemma 5]{mozgovoy_computational} that $A_d^\ind(q)$ is also a polynomial in $q$ and
\[\sum_{d\ge0}M_d(q)t^d=\Exp\rbr{\sum_{d\ge1}A_d^\ind(q)t^d}.\]
Applying the previous theorem, we obtain $\sum_{d\ge1}A_d^\ind(q)t^d=$
\[=\Log\Pow\rbr{\sum_{\la}r_\la(q)^{m-1}t^{\n\la},q-1}=(q-1)\Log\rbr{\sum_{\la}r_\la(q)^{m-1}t^{\n\la}}.\]
\end{proof}
\section{Counting subgroups}\label{subgroups}
The aim of this section is to relate the counting polynomials $A_d^\irr(q)$ with the number of index $d$ subgroups of $\Ga=F_m$. Such a relation is motivated by ideas of $\bF_1$-geometry \cite{lopezpena_mapping}: viewing the symmetric group $S_n$ as ${\rm GL}_n$ over a hypothetical field with one element, one can expect counts of permutation representations of a group, or, equivalently, (conjugacy classes of) subgroups of finite index, to arise from the counts of representations over fields with $q$ elements by an appropriate limit process $q\rightarrow 1$. This is accomplished in Lemma \ref{limitsgg}.

\subsection{Subgroups and permutation representations}
Let $G$ be a finitely generated group. A (permutation) representation of $G$ of order $n$ is an action of $G$ on $[n]=\set{1,\dots,n}$. It is called irreducible if the action of $G$ is transitive. One can define the notion of an isomorphism between two representations in a natural way.
Any irreducible order $n$ representation can be written in the form $G/H$, where $H$ is an index $n$ subgroup of $G$. Two irreducible representations $G/H$ and $G/H'$ are isomorphic if and only if $H$ and $H'$ are conjugate. Moreover
\[\Aut(G/H)\iso N_GH/H.\]

Let $J_n$ denote the set of index $n$ subgroups and $I_n$ denote the set of conjugacy classes of index $n$ subgroups (or isoclasses of irreducible order $n$ representations). Given a representation $M=G/H$ in $I_n$, the number of elements in the conjugacy class of $H$ is equal to
\[\n{G/N_GH}=\frac{\n{G/H}}{\n{N_GH/H}}=\frac{n}{\n{\Aut M}}.\]
Therefore
\eql{\frac{\n{J_n}}n=\sum_{M\in I_n}\frac{1}{\n{\Aut M}}.}{J as weighted count}

\begin{remark}%\label{rk:}
We should call an irreducible representation $G/H$ absolutely irreducible if $\n{\Aut(G/H)}=1$, that is, $N_GH=H$. This means that $\n{J_n}/n$ counts all irreducible representations and not just the absolutely irreducible.
\end{remark}

\subsection{Two formulas}
The set $R_n=\Hom(G,S_n)$ is equipped with an action of $S_n$ by conjugation. Let $R_n^\irr\sb R_n$ be the set of irreducible representations. Then
\[I_n=R_n^\irr/S_n.\]
According to \cite[Ex.~5.13]{stanley_enumerative2}
\eql{\sum_{n\ge0}\sum_{M\in R_n/S_n}\frac{t^n}{\n{\Aut M}}
=\sum_{n\ge0}\frac{\n{R_n}}{\n{S_n}}t^n
=\exp\rbr{\sum_{n\ge1}\sum_{M\in I_n}\frac{t^n}{\n{\Aut M}}},}{stan1}
\eql{\sum_{n\ge0}\n{R_n/S_n}t^n=\Exp\rbr{\sum_{n\ge1}\n{I_n}t^n}.}{stan2}
The second formula is a standard relation between all representations and indecomposable representations.

\begin{remark}%\label{rk:}
The second formula is equivalent to the statement of \cite[Ex.~5.13.c]{stanley_enumerative2}, because
\[\Hom(G\xx\bZ,S_n)=\sets{(M,g)\in R_n\xx S_n}{g\in\Aut M}\]
and therefore
\[\frac{\n{\Hom(G\xx\bZ,S_n)}}{\n{S_n}}=\sum_{M\in R_n}\frac{\n{\Aut M}}{\n{S_n}}=\n{R_n/S_n}.\]
\end{remark}

\begin{remark}%\label{rk:}
One can define the "formal" Hall algebra for permutation representations as follows. Let $J=\bigcup_{n\ge1}I_n$ be the set of isoclasses of all irreducible representations. Each $S\in I_n$ has two integer parameters $\dim S=n$ and $\n{\Aut S}$.
All representations are parametrized by maps $m:J\to\bN$ with finite support. Note that
\[\n{\Aut m}=\prod_{S\in J}m_S!\cdot \n{\Aut S}^{m_S},\]
while in the usual semisimple category $\n{\Aut m}=\prod_{S\in J}\n{\GL_{m_S}(\End S)}$. 
The "formal" Hall algebra $H$ has these maps as a basis and has multiplication given by
\[m\circ n=\prod_{S\in J}\binom{m_S+n_S}{m_S}\cdot [m\oplus n].\]
There is an integration map $I:H\to\bQ\pser t$ (the latter algebra has the usual multiplication) given by
\[m\mto \frac{t^{\dim m}}{\n{\Aut m}}
=\prod_{S\in J}\frac{t^{m_S\dim S}}{m_S!\n{\Aut S}^{m_S}}.\]
Formula \eqref{stan1} can be obtained by analyzing this integration map.
\end{remark}

\subsection{The case
\texorpdfstring{$G=F_m$}{G=Fm}}
The polynomials $A_d^{\irr}(q)$ count absolutely irreducible representations, while $\n{J_n}/n$ counts all irreducible permutation representations.
There is a straightforward relation between these counts:

\begin{lemma}\label{limitsgg} We have
\[\Psi\rbr{\sum_{n\ge1}\frac{A_n^{\irr}(q)}{q-1}t^n}_{t=\frac{x}{(q-1)^{m-1}},q=1}=\sum_{n\ge1}\frac{\n{J_n}}{n}x^n.\]
\end{lemma}

\begin{proof} This follows from
\[\Psi\rbr{\sum_{n\ge1}\frac{A_n^{\irr}(q)}{1-q}t^n}
=\log\rbr{T\inv\rbr{\sum_{n\ge0}[n]_q^{!(m-1)}((q-1)^{m-1}t)^n}\inv},\]
\[\sum_{n\ge1}\frac{\n{J_n}}{n}x^n=\log\rbr{\sum_{n\ge0}n!^{m-1}x^n}.\]
\end{proof}

An explanation for this relation could be that in order to pass from absolutely irreducible representations to all irreducible representations we have to apply Adams operations as in \cite[Lemma 3.2]{mozgovoy_number}.
%However, it is not quite clear why we should apply the transformation $t\mto t/(q-1)^{m-1}$.

%\input{diff.tex}

%\bibliography{../tex/papers}

\begin{thebibliography}{10}

\bibitem{bridgeland_introduction}
Tom Bridgeland, \emph{{A}n introduction to motivic {H}all algebras}, Adv. Math.
  \textbf{229} (2012), no.~1, 102--138,
  \href{http://arxiv.org/abs/1002.4372}{{\ttfamily arXiv:1002.4372}}.

\bibitem{butler_number}
Steve Butler, \emph{{T}he number of ordered tuples with no global factor},
  2006, Preprint.

\bibitem{cavazos_e}
Samuel Cavazos and Sean Lawton, \emph{${E}$-polynomial of
  $\mathrm{SL}_2(\mathbb{C})$-{C}haracter {V}arieties of {F}ree groups}, 2014,
  \href{http://arxiv.org/abs/1401.0228}{{\ttfamily arXiv:1401.0228}}.

\bibitem{florentino_topology}
Carlos Florentino and Sean Lawton, \emph{{T}he topology of moduli spaces of
  free group representations}, Math. Ann. \textbf{345} (2009), no.~2, 453--489.

\bibitem{florentino_singularities}
\bysame, \emph{{S}ingularities of free group character varieties}, Pacific J.
  Math. \textbf{260} (2012), no.~1, 149--179.

\bibitem{hall_subgroups}
Marshall Hall, Jr., \emph{{S}ubgroups of finite index in free groups}, Canadian
  J. Math. \textbf{1} (1949), 187--190.

\bibitem{hausel_mixed}
Tam{\'a}s Hausel and Fernando Rodriguez-Villegas, \emph{{M}ixed {H}odge
  polynomials of character varieties}, Invent. Math. \textbf{174} (2008),
  no.~3, 555--624, \href{http://arxiv.org/abs/math/0612668}{{\ttfamily
  arXiv:math/0612668}}, With an appendix by Nicholas M. Katz.

\bibitem{hua_counting}
Jiuzhao Hua, \emph{{C}ounting representations of quivers over finite fields},
  J. Algebra \textbf{226} (2000), no.~2, 1011--1033.

\bibitem{lopezpena_mapping} Javier L{\'o}pez Pe{\~n}a and Oliver Lorscheid, \emph{Mapping {$\mathbb F_1$}-land: an overview of geometries over
              the field with one element}, Noncommutative geometry, arithmetic, and related topics, 241--265, Johns Hopkins Univ. Press, Baltimore, 2011.

\bibitem{lubotzky_subgroup} Alexander Lubotzky and Dan Segal, \emph{Subgroup growth}, Progress in Mathematics \textbf{212}, Birkh\"auser, Basel 2003.

\bibitem{macdonald_symmetric}
I.~G. Macdonald, \emph{{S}ymmetric functions and {H}all polynomials}, second
  ed., Oxford Mathematical Monographs, Oxford University Press, 1995, With
  contributions by A. Zelevinsky.

\bibitem{mozgovoy_computational}
Sergey Mozgovoy, \emph{{A} computational criterion for the {K}ac conjecture},
  J. Algebra \textbf{318} (2007), no.~2, 669--679,
  \href{http://arxiv.org/abs/math/0608321}{{\ttfamily arXiv:math/0608321}}.

\bibitem{mozgovoy_motivicb}
\bysame, \emph{{M}otivic {D}onaldson-{T}homas invariants and {M}c{K}ay
  correspondence}, 2011, \href{http://arxiv.org/abs/1107.6044}{{\ttfamily
  arXiv:1107.6044}}.

\bibitem{mozgovoy_number}
Sergey Mozgovoy and Markus Reineke, \emph{{O}n the number of stable quiver
  representations over finite fields}, J. Pure Appl. Algebra \textbf{213}
  (2009), no.~4, 430--439, \href{http://arxiv.org/abs/0708.1259}{{\ttfamily
  arXiv:0708.1259}}.

\bibitem{seshadri_geometric}
C.~S. Seshadri, \emph{{G}eometric reductivity over arbitrary base}, Advances in
  Math. \textbf{26} (1977), no.~3, 225--274.

\bibitem{stanley_enumerative2}
Richard~P. Stanley, \emph{{E}numerative combinatorics. {V}ol. 2}, Cambridge
  Studies in Advanced Mathematics, vol.~62, Cambridge University Press,
  Cambridge, 1999, With a foreword by Gian-Carlo Rota and appendix 1 by Sergey
  Fomin.

\end{thebibliography}
%\bibliographystyle{../tex/hamsplain}
 
\providecommand{\bysame}{\leavevmode\hbox to3em{\hrulefill}\thinspace}
\providecommand{\href}[2]{#2}

\end{document}